\documentclass[12pt]{article}
\usepackage{amsmath,amsfonts,amssymb,amsthm,txfonts,hyperref,url,t1enc}
\newtheorem{theorem}{Theorem}
\newtheorem{lemma}{Lemma}
\newcommand{\N}{\mathbb N}
\newcommand{\Z}{\mathbb Z}

\newcommand{\K}{\textbf{\textit{K}}}
\begin{document}
\par
\noindent
\centerline{{\Large Small systems of Diophantine equations which have}}
\vskip 0.2truecm
\par
\noindent
\centerline{{\Large only very large integer solutions}}
\vskip 1.1truecm
\par
\noindent
\centerline{{\Large Apoloniusz Tyszka}}
\vskip 1.1truecm
\noindent
{\bf Abstract.} Let \mbox{$E_n=\{x_i=1,~x_i+x_j=x_k,~x_i \cdot x_j=x_k: i,j,k \in \{1,\ldots,n\}\}$}.
There is an algorithm that for every computable function
\mbox{$f:\N \to \N$} returns a positive integer $m(f)$, for which a second algorithm accepts
on the input $f$ and any integer \mbox{$n \geq m(f)$}, and returns a system \mbox{$S \subseteq E_n$} such that
$S$ has infinitely many integer solutions and each integer tuple $(x_1,\ldots,x_n)$ that solves $S$ satisfies $x_1=f(n)$.
For each integer \mbox{$n \geq 12$} we construct a system \mbox{$S \subseteq E_n$}
such that~$S$ has infinitely many integer solutions and they all belong to
\mbox{${\Z}^n \setminus [-2^{\textstyle 2^{n-1}},2^{\textstyle 2^{n-1}}]^n$}.
\vskip 1.1truecm
\par
We present a general method for constructing small systems
of Diophantine equations which have only very large integer solutions. Let $\Phi_n$ denote the following statement
\[
\forall x_1,\ldots,x_n \in \Z ~\exists y_1,\ldots,y_n \in \Z
\]
\[
\Bigl(2^{\textstyle 2^{n-1}}<|x_1| \Longrightarrow \bigl(|x_1|<|y_1| \vee \ldots \vee |x_1|<|y_n|\bigr)\Bigr) ~\wedge
\]
\begin{equation}
\Bigl(\forall i,j,k \in \{1,\ldots,n\}~(x_i+x_j=x_k \Longrightarrow y_i+y_j=y_k)\Bigr) ~\wedge
\end{equation}
\begin{equation}
\forall i,j,k \in \{1,\ldots,n\}~(x_i \cdot x_j=x_k \Longrightarrow y_i \cdot y_j=y_k)
\end{equation}
\par
For $n \geq 2$, the bound $2^{\textstyle 2^{n-1}}$ cannot be decreased because for
\[
(x_1,\ldots,x_n)=\Bigl(2^{\textstyle 2^{n-1}},2^{\textstyle 2^{n-2}},2^{\textstyle 2^{n-3}},\ldots,256,16,4,2\Bigr)
\]
the conjunction of statements (1) and (2) guarantees that
\[
(y_1,\ldots,y_n)=(0,\ldots,0) \vee (y_1,\ldots,y_n)=\Bigl(2^{\textstyle 2^{n-1}},2^{\textstyle 2^{n-2}},2^{\textstyle 2^{n-3}},\ldots,256,16,4,2\Bigr)
\]
\par
The statement $\forall n \Phi_n$ has powerful consequences for Diophantine equations, but is still unproven, see \cite{Tyszka}.
In particular, it implies that if a Diophantine equation has only finitely many solutions
in integers (non-negative integers, rationals), then their heights are bounded from above
by a computable function of the degree and the coefficients of the equation.
For integer solutions, this conjectural upper bound can be computed by applying equation~(3) and Lemmas~\ref{lem2} and~\ref{lem7}.
\vskip 0.2truecm
\begin{sloppypar}
\noindent
{\bf Observation.} {\em For all positive integers $n$, $m$ with \mbox{$n \leq m$},
if the statement $\Phi_n$ fails for \mbox{$(x_1,\ldots,x_n) \in {\Z}^n$}
and \mbox{$2^{\textstyle 2^{m-1}}<|x_1| \leq 2^{\textstyle 2^m}$}, then the statement $\Phi_m$ fails
for \mbox{$(\underbrace{x_1,\ldots,x_1}_{m-n+1 {\rm ~times}},x_2,\ldots,x_n) \in {\Z}^m$}.}
\end{sloppypar}
\vskip 0.2truecm
\par
By the Observation, the statement $\forall n \Phi_n$ is equivalent to the statement $\forall n \Psi_n$,
where $\Psi_n$ denote the statement
\[
\forall x_1,\ldots,x_n \in \Z ~\exists y_1,\ldots,y_n \in \Z
\]
\[
\Bigl(2^{\textstyle 2^{n-1}}<|x_1|={\rm max}\bigl(|x_1|,\ldots,|x_n|\bigr) \leq 2^{\textstyle 2^n} \Longrightarrow \bigl(|x_1|<|y_1| \vee \ldots \vee |x_1|<|y_n|\bigr)\Bigr) ~\wedge
\]
\[
\Bigl(\forall i,j,k \in \{1,\ldots,n\}~(x_i+x_j=x_k \Longrightarrow y_i+y_j=y_k)\Bigr) ~\wedge
\]
\[
\forall i,j,k \in \{1,\ldots,n\}~(x_i \cdot x_j=x_k \Longrightarrow y_i \cdot y_j=y_k)
\]
In contradistinction to the statements $\Phi_n$, each true statement $\Psi_n$
can be confirmed by a brute-force search in a finite amount of time.
\vskip 0.2truecm
\par
The statement
\[
\forall n ~\forall x_1,\ldots,x_n \in \Z ~\exists y_1,\ldots,y_n \in \Z
\]
\[
\bigl(2^{\textstyle 2^{n-1}}<|x_1| \Longrightarrow |x_1|<|y_1|\bigr) ~\wedge
\]
\[
\bigl(\forall i,j,k \in \{1,\ldots,n\}~(x_i+x_j=x_k \Longrightarrow y_i+y_j=y_k)\bigr) ~\wedge
\]
\[
\forall i,j,k \in \{1,\ldots,n\}~(x_i \cdot x_j=x_k \Longrightarrow y_i \cdot y_j=y_k)
\]
strengthens the statement $\forall n \Phi_n$ but is false, as we will show in the Corollary.
\vskip 0.2truecm
\par
Let
\[
E_n=\{x_i=1,~x_i+x_j=x_k,~x_i \cdot x_j=x_k: i,j,k \in \{1,\ldots,n\}\}
\]
\newpage
To each system $S \subseteq E_n$ we assign the system $\widetilde{S}$ defined by
\vskip 0.2truecm
\par
\noindent
\centerline{$\left(S \setminus \{x_i=1:~i \in \{1,\ldots,n\}\}\right) \cup$}
\par
\noindent
\centerline{$\{x_i \cdot x_j=x_j:~i,j \in \{1,\ldots,n\} {\rm ~and~the~equation~} x_i=1 {\rm ~belongs~to~} S\}$}
\vskip 0.2truecm
\par
\noindent
In other words, in order to obtain $\widetilde{S}$ we remove from $S$ each
equation $x_i=1$ and replace it by the following $n$ equations:
\vskip 0.2truecm
\par
\noindent
\centerline{$\begin{array}{rcl}
x_i \cdot x_1 &=& x_1\\
&\ldots& \\
x_i \cdot x_n &=& x_n
\end{array}$}
\vskip 0.2truecm
\par
\noindent
\begin{lemma}\label{lem1}
For each system $S \subseteq E_n$
\begin{eqnarray*}
\{(x_1,\ldots,x_n) \in {\Z}^n:~(x_1,\ldots,x_n) {\rm ~solves~} \widetilde{S}\} &=& \\
\{(x_1,\ldots,x_n) \in {\Z}^n:~(x_1,\ldots,x_n) {\rm ~solves~} S\} \cup
\{(0,\ldots,0)\}&
\end{eqnarray*}
\end{lemma}
\begin{lemma}\label{lem2}
The statement $\Phi_n$ can be equivalently stated thus: if a \mbox{system}
\mbox{$S \subseteq E_n$} has only finitely many solutions in integers \mbox{$x_1,\ldots,x_n$}, then each such
\mbox{solution} \mbox{$(x_1,\ldots,x_n)$} satisfies \mbox{$|x_1|,\ldots,|x_n| \leq 2^{\textstyle 2^{n-1}}$}.
\end{lemma}
\begin{proof}
It follows from Lemma~\ref{lem1}.
\end{proof}
\par
Nevertheless, for each integer \mbox{$n \geq 12$} there exists a system \mbox{$S \subseteq E_n$} which has infinitely many
integer solutions and they all belong to \mbox{${\Z}^n \setminus [-2^{\textstyle 2^{n-1}},2^{\textstyle 2^{n-1}}]^n$}.
We will prove it in Theorem~\ref{the1}. First we need a few lemmas.
\begin{lemma}\label{lem3}
If a positive integer $n$ is odd and a pair $(x,y)$ of positive integers solves the negative Pell equation \mbox{$x^2-dy^2=-1$},
then the pair
\[
\left(\frac{\left(x+y\sqrt{d}\right)^n+\left(x-y\sqrt{d}\right)^n}{2},~
\frac{\left(x+y\sqrt{d}\right)^n-\left(x-y\sqrt{d}\right)^n}{2\sqrt{d}}\right)
\]
consists of positive integers and solves the equation \mbox{$x^2-dy^2=-1$}.
\end{lemma}
\begin{lemma}\label{lem4} \mbox{(\cite[pp.~201--202,~Theorem~106]{Nagell})} In the domain of positive integers,
all solutions to \mbox{$x^2-5y^2=-1$} are given by
\[
\left(2+\sqrt{5}\right)^{2k+1}=x+y\sqrt{5}
\]
where $k$ is a non-negative integer.
\end{lemma}
\begin{lemma}\label{lem5}
The pair $(2,1)$ solves the equation $x^2-5y^2=-1$. If a pair \mbox{$(x,y)$} solves the
equation \mbox{$x^2-5y^2=-1$}, then the pair \mbox{$(9x+20y,~4x+9y)$} solves this equation too.
\end{lemma}
\begin{lemma}\label{lem6}
\begin{sloppypar}
Lemma~\ref{lem5} allows us to compute all positive integer solutions to \mbox{$x^2-5y^2=-1$}.
\end{sloppypar}
\end{lemma}
\begin{proof}
It follows from Lemma~\ref{lem4}. Indeed, if $\left(2+\sqrt{5}\right)^{2k+1}=x+y\sqrt{5}$, then
\[
\left(2+\sqrt{5}\right)^{2k+3}=\left(2+\sqrt{5}\right)^2 \cdot \left(2+\sqrt{5}\right)^{2k+1}=
\]
\[
\left(9+4\sqrt{5}\right) \cdot \left(x+y\sqrt{5}\right)=\left(9x+20y\right)+\left(4x+9y\right)\sqrt{5}
\]
\end{proof}
\begin{theorem}\label{the1}
For each integer $n \geq 12$ there exists a system \mbox{$S~\subseteq~E_n$} such that $S$ has infinitely
many integer solutions and they all belong to \mbox{${\Z}^n \setminus [-2^{\textstyle 2^{n-1}},2^{\textstyle 2^{n-1}}]^{n}$}.
\end{theorem}
\begin{proof}
By Lemmas~\ref{lem4}--\ref{lem6}, the equation \mbox{$u^2-5v^2=-1$} has infinitely many solutions in positive
integers and all these solutions can be simply computed. For a positive integer~$n$, let \mbox{$(u(n),v(n))$} denote
the \mbox{$n$-th} solution to \mbox{$u^2-5v^2=-1$}. We define $S$ as
\vskip 0.2truecm
\par
\noindent
\centerline{$x_1=1$~~~~~~~~~~$x_1+x_1=x_2$~~~~~~~~~~$x_2+x_2=x_3$~~~~~~~~~~$x_1+x_3=x_4$}
\vskip 0.2truecm
\par
\noindent
\centerline{$x_4 \cdot x_4=x_5$~~~~~~~~~~$x_5 \cdot x_5=x_6$~~~~~~~~~~$x_6 \cdot x_7=x_8$~~~~~~~~~~$x_8 \cdot x_8=x_9$}
\vskip 0.2truecm
\par
\noindent
\centerline{$x_{10} \cdot x_{10}=x_{11}$~~~~~~~~~~$x_{11}+x_1=x_{12}$~~~~~~~~~~$x_4 \cdot x_9=x_{12}$}
\vskip 0.2truecm
\par
\noindent
\centerline{$x_{12} \cdot x_{12}=x_{13}$~~~~~~~~~~$x_{13} \cdot x_{13}=x_{14}$~~~~~~~~~~\ldots~~~~~~~~~~$x_{n-1} \cdot x_{n-1}=x_n$}
\vskip 0.2truecm
\noindent
The first $11$ equations of $S$ equivalently expresses that \mbox{$x_{10}^2-5 \cdot x_8^2=-1$} and 625 divides $x_8$.
The equation {$x_{10}^2-5^9 \cdot x_7^2=-1$} expresses the same fact. Execution of the following {\sl MuPAD} code
\begin{quote}
\begin{verbatim}
x:=2:
y:=1:
for n from 2 to 313 do
u:=9*x+20*y:
v:=4*x+9*y:
if igcd(v,625)=625 then print(n) end_if:
x:=u:
y:=v:
end_for:
float(u^2+1);
float(2^(2^(12-1)));
\end{verbatim}
\end{quote}
\begin{sloppypar}
\noindent
returns only $n=313$. Therefore, in the domain of positive integers, the \mbox{minimal} solution to \mbox{$x_{10}^2-5^9 \cdot x_7^2=-1$}
is given by the pair \mbox{$\left(x_{10}=u(313),~x_7=\frac{\textstyle v(313)}{\textstyle 625}\right)$}.
Hence, if an integer tuple \mbox{$(x_1,\ldots,x_n)$} solves $S$, then \mbox{$|x_8| \geq v(313)$} and
\[
x_{12}=x_{10}^2+1 \geq u(313)^2+1>2^{\textstyle 2^{12-1}}
\]
The final inequality comes from the execution of the last two \mbox{instructions} of the code,
as they display the numbers \mbox{$1.263545677e783$} and \mbox{$3.231700607e616$}.
Applying induction, we get \mbox{$x_n>2^{\textstyle 2^{n-1}}$}. By Lemma~\ref{lem3} (or by \mbox{\cite[p.~58,~Theorem~1.3.6]{Yan})},
the equation \mbox{$x_{10}^2-5^9 \cdot x_7^2=-1$} has infinitely many \mbox{integer} solutions. This conclusion transfers to the \mbox{system~$S$}.
\end{sloppypar}
\end{proof}
\begin{sloppypar}
J.~C.~Lagarias studied the equation \mbox{$x^2-dy^2=-1$} for \mbox{$d=5^{2n+1}$},
where \mbox{$n=0,1,2,3,\ldots$}. His theorem says that for these values \mbox{of~$d$},
the least integer solution grows exponentially \mbox{with~$d$}, \mbox{see \cite[Appendix~A]{Lagarias}}.
\vskip 0.2truecm
\par
The next theorem generalizes Theorem~\ref{the1}. But first we need Lemma~\ref{lem7} together with introductory matter.
\vskip 0.2truecm
\par
Let \mbox{$D(x_1,\ldots,x_p) \in {\Z}[x_1,\ldots,x_p]$}. For the Diophantine equation \mbox{$2 \cdot D(x_1,\ldots,x_p)=0$},
let $M$ denote the maximum of the absolute values of its coefficients.
Let ${\cal T}$ denote the family of all polynomials
$W(x_1,\ldots,x_p) \in {\Z}[x_1,\ldots,x_p]$ whose all coefficients belong to the interval $[-M,M]$
and ${\rm deg}(W,x_i) \leq d_i={\rm deg}(D,x_i)$ for each $i \in \{1,\ldots,p\}$.
Here we consider the degrees of $W(x_1,\ldots,x_p)$ and $D(x_1,\ldots,x_p)$
with respect to the variable~$x_i$. It is easy to check that
\begin{equation}
{\rm card}({\cal T})=(2M+1)^{\textstyle (d_1+1) \cdot \ldots \cdot (d_p+1)}
\end{equation}
\par
We choose any bijection \mbox{$\tau: \{p+1,\ldots,{\rm card}({\cal T})\} \longrightarrow {\cal T} \setminus \{x_1,\ldots,x_p\}$}.
Let ${\cal H}$ denote the family of all equations of the form
\vskip 0.2truecm
\noindent
\centerline{$x_i=1$, $x_i+x_j=x_k$, $x_i \cdot x_j=x_k$~~($i,j,k \in \{1,\ldots,{\rm card}({\cal T})\})$}
\vskip 0.2truecm
\noindent
which are polynomial identities in \mbox{${\Z}[x_1,\ldots,x_p]$} if
\[
\forall s \in \{p+1,\ldots,{\rm card}({\cal T})\} ~~x_s=\tau(s)
\]
There is a unique \mbox{$q \in \{p+1,\ldots,{\rm card}({\cal T})\}$} such that \mbox{$\tau(q)=2 \cdot D(x_1,\ldots,x_p)$}.
For each ring $\K$ extending $\Z$ the system ${\cal H}$ implies \mbox{$2 \cdot D(x_1,\ldots,x_p)=x_q$}.
To see this, we observe that there exist pairwise distinct
\mbox{$t_0,\ldots,t_m \in {\cal T}$} such that $m>p$ and
\[
t_0=1~ \wedge ~t_1=x_1~ \wedge ~\ldots~ \wedge ~t_p=x_p~ \wedge ~t_m=2 \cdot D(x_1,\ldots,x_p)~ \wedge
\]
\[
\forall i \in \{p+1,\ldots,m\}~ \exists j,k \in \{0,\ldots,i-1\} ~~(t_j+t_k=t_i \vee t_i+t_k=t_j \vee t_j \cdot t_k=t_i)
\]
For each ring $\K$ extending $\Z$ and for each \mbox{$x_1,\ldots,x_p \in \K$}
there exists a unique tuple \mbox{($x_{p+1},\ldots,x_{{\rm card}({\cal T})}) \in \K^{{\rm card}({\cal T})-p}$}
such that the tuple \mbox{$(x_1,\ldots,x_p,x_{p+1},\ldots,x_{{\rm card}({\cal T})})$}
solves the system \mbox{${\cal H}$}. The sought elements \mbox{$x_{p+1},\ldots,x_{{\rm card}({\cal T})}$}
are given by the formula
\[
\forall s \in \{p+1,\ldots,{\rm card}({\cal T})\} ~~x_s=\tau(s)(x_1,\ldots,x_p)
\]
\begin{lemma}\label{lem7}
The system ${\cal H} \cup \{x_q+x_q=x_q\}$ can be simply computed.
For each ring $\K$ extending $\Z$, the equation $D(x_1,\ldots,x_p)=0$
is equivalent to the system ${\cal H} \cup \{x_q+x_q=x_q\} \subseteq E_{{\rm card}({\cal T})}$.
Formally, this equivalence can be written as
\[
\forall x_1,\ldots,x_p \in \K ~\Bigl(D(x_1,\ldots,x_p)=0 \Longleftrightarrow
\exists x_{p+1},\ldots,x_{{\rm card}({\cal T})} \in \K
\]
\[
(x_1,\ldots,x_p,x_{p+1},\ldots,x_{{\rm card}({\cal T})}) {\rm ~solves~the~system~}
{\cal H} \cup \{x_q+x_q=x_q\} \Bigr)
\]
For each ring $\K$ extending $\Z$ and for each \mbox{$x_1,\ldots,x_p \in \K$} with
\mbox{$D(x_1,\ldots,x_p)=0$} there exists a unique tuple
\mbox{($x_{p+1},\ldots,x_{{\rm card}({\cal T})}) \in \K^{{\rm card}({\cal T})-p}$} such
that the tuple \mbox{$(x_1,\ldots,x_p,x_{p+1},\ldots,x_{{\rm card}({\cal T})})$} solves the system
\mbox{${\cal H} \cup \{x_q+x_q=x_q\}$}. Hence, for each ring $\K$ extending $\Z$ the equation
\mbox{$D(x_1,\ldots,x_p)=0$} has the same number of solutions as the system \mbox{${\cal H} \cup \{x_q+x_q=x_q\}$}.
\end{lemma}
\end{sloppypar}
\par
Putting $M=M/2$ we obtain new families ${\cal T}$ and ${\cal H}$.
There is a unique $q \in \{1,\ldots,{\rm card}({\cal T})\}$ such that
\[
\Bigl(q \in \{1,\ldots,p\}~ \wedge ~x_q=D(x_1,\ldots,x_p)\Bigr)~ \vee
\]
\[
\Bigl(q \in \{p+1,\ldots,{\rm card}({\cal T})\}~ \wedge ~\tau(q)=D(x_1,\ldots,x_p)\Bigr)
\]
The new system \mbox{${\cal H} \cup \{x_q+x_q=x_q\}$} is equivalent to \mbox{$D(x_1,\ldots,x_p)=0$}
and can be simply computed.
\newpage
The Davis-Putnam-Robinson-Matiyasevich theorem states that every recursively enumerable
set \mbox{${\cal M} \subseteq {\N}^n$} has a Diophantine representation, that is
\[
(a_1,\ldots,a_n) \in {\cal M} \Longleftrightarrow
\exists x_1, \ldots, x_m \in \N ~~W(a_1,\ldots,a_n,x_1,\ldots,x_m)=0
\]
\par
\noindent
for some polynomial $W$ with integer coefficients, see \cite{Matiyasevich} and \cite{Kuijer}.
The polynomial~$W$ can be computed, if we know a Turing machine~$M$
such that, for all \mbox{$(a_1,\ldots,a_n) \in {\N}^n$}, $M$ halts on \mbox{$(a_1,\ldots,a_n)$} if and only
if \mbox{$(a_1,\ldots,a_n) \in {\cal M}$}, see \cite{Matiyasevich} and \cite{Kuijer}.
\begin{theorem}\label{the2}
There is an algorithm that for every computable function
\mbox{$f:\N \to \N$} returns a positive integer $m(f)$, for which a second algorithm accepts
on the \mbox{input $f$} and any integer \mbox{$n \geq m(f)$}, and returns a system \mbox{$S \subseteq E_n$} such that
$S$ has infinitely many integer solutions and each integer tuple $(x_1,\ldots,x_n)$ that solves $S$ satisfies $x_1=f(n)$.
\end{theorem}
\begin{proof}
\begin{sloppypar}
\noindent
By the Davis-Putnam-Robinson-Matiyasevich theorem, the function $f$ has a Diophantine representation.
It means that there is a polynomial $W(x_1,x_2,x_3,\ldots,x_r)$ with integer coefficients
such that for each non-negative integers $x_1$, $x_2$,
\begin{equation}
\tag*{\tt (E1)}
x_1=f(x_2) \Longleftrightarrow \exists x_3, \ldots, x_r \in \N ~~W(x_1,x_2,x_3,\ldots,x_r)=0
\end{equation}
\end{sloppypar}
\noindent
By the equivalence~{\rm (E1)} and Lagrange's four-square theorem, for each integers $x_1$, $x_2$,
the conjunction \mbox{$(x_2 \geq 0) \wedge (x_1=f(x_2))$} holds true if and only if there exist integers
$a,b,c,d,\alpha,\beta,\gamma,\delta,x_3,x_{3,1},x_{3,2},x_{3,3},x_{3,4},\ldots,x_r,x_{r,1},x_{r,2},x_{r,3},x_{r,4}$
such that
\[
W^2(x_1,x_2,x_3,\ldots,x_r)+\bigl(x_1-a^2-b^2-c^2-d^2\bigr)^2+\bigl(x_2-\alpha^2-\beta^2-\gamma^2-\delta^2\bigr)^2+
\]
\[
\bigl(x_3-x^2_{3,1}-x^2_{3,2}-x^2_{3,3}-x^2_{3,4}\bigr)^2+\ldots+\bigl(x_r-x^2_{r,1}-x^2_{r,2}-x^2_{r,3}-x^2_{r,4}\bigr)^2=0
\]
By Lemma~\ref{lem7}, there is an integer \mbox{$s \geq 3$} such that for each integers $x_1$, $x_2$,
\begin{equation}
\tag*{\tt (E2)}
\Bigl(x_2 \geq 0 \wedge x_1=f(x_2)\Bigr) \Longleftrightarrow \exists x_3,\ldots,x_s \in \Z ~~\Psi(x_1,x_2,x_3,\ldots,x_s)
\end{equation}
where the formula $\Psi(x_1,x_2,x_3,\ldots,x_s)$ is algorithmically determined as a conjunction of formulae of the form
\mbox{$x_i=1$}, \mbox{$x_i+x_j=x_k$}, \mbox{$x_i \cdot x_j=x_k$} \mbox{($i,j,k \in \{1,\ldots,s\})$}.
Let $m(f)=8+2s$, and let $[\cdot]$ denote the integer part function. For each integer \mbox{$n \geq m(f)$},
\[
n-\left[\frac{n}{2}\right]-4-s \geq m(f)-\left[\frac{m(f)}{2}\right]-4-s \geq m(f)-\frac{m(f)}{2}-4-s=0
\]
Let $S$ denote the following system
\[\left\{
\begin{array}{rcl}
{\rm all~equations~occurring~in~}\Psi(x_1,x_2,\ldots,x_s) \\
n-\left[\frac{n}{2}\right]-4-s {\rm ~equations~of~the~form~} z_i=1 \\
t_1 &=& 1 \\
t_1+t_1 &=& t_2 \\
t_2+t_1 &=& t_3 \\
&\ldots& \\
t_{\left[\frac{n}{2}\right]-1}+t_1 &=& t_{\left[\frac{n}{2}\right]} \\
t_{\left[\frac{n}{2}\right]}+t_{\left[\frac{n}{2}\right]} &=& w \\
w+y &=& x_2 \\
y+y &=& y {\rm ~(if~}n{\rm ~is~even)} \\
y &=& 1 {\rm ~(if~}n{\rm ~is~odd)} \\
u+u &=& v
\end{array}
\right.\]
with $n$ variables. By the equivalence~{\tt (E2)}, the system~$S$ is consistent over $\Z$.
The equation \mbox{$u+u=v$} guarantees that $S$ has infinitely many integer solutions.
If an integer $n$-tuple $(x_1,x_2,\ldots,x_s,\ldots,w,y,u,v)$ solves~$S$,
then by the equivalence~{\tt (E2)},
\[
x_1=f(x_2)=f(w+y)=f\left(2 \cdot \left[\frac{n}{2}\right]+y\right)=f(n)
\]
\end{proof}
\vskip 0.2truecm
\noindent
{\bf Corollary.} {\em There is an algorithm that for every computable function
\mbox{$f:\N \to \N$} returns a positive integer $m(f)$, for which a second algorithm accepts
on the \mbox{input $f$} and any integer \mbox{$n \geq m(f)$}, and returns an integer tuple \mbox{$(x_1,\ldots,x_n)$}
for which $x_1=f(n)$ and
\vskip 0.2truecm
\par
\noindent
~~~~~(4)~~for each integers $y_1,\ldots,y_n$ the conjunction
\[
\Bigl(\forall i \in \{1,\ldots,n\}~(x_i=1 \Longrightarrow y_i=1)\Bigr) ~\wedge
\]
\[
\Bigl(\forall i,j,k \in \{1,\ldots,n\}~(x_i+x_j=x_k \Longrightarrow y_i+y_j=y_k)\Bigr) ~\wedge
\]
\[
\forall i,j,k \in \{1,\ldots,n\}~(x_i \cdot x_j=x_k \Longrightarrow y_i \cdot y_j=y_k)
\]
\par
\noindent
~~~~~~~~~~~~implies that $x_1=y_1$.}
\vskip 0.2truecm
\begin{sloppypar}
\noindent
{\em Proof.} Let \mbox{$\leq_n$} denote the order on \mbox{${\Z}^n$} which ranks the tuples \mbox{$(x_1,\ldots,x_n)$}
first \mbox{according} to \mbox{${\rm max}(|x_1|,\ldots,|x_n|)$} and then lexicographically.
The ordered set \mbox{$({\Z}^n,\leq_n)$} is isomorphic to \mbox{$(\N,\leq)$}. To find an integer tuple
\mbox{$(x_1,\ldots,x_n)$}, we solve the system~$S$ by performing the brute-force search in the order~$\leq_n$.
\end{sloppypar}
\noindent
\rightline{$\Box$}
\vskip 0.2truecm
\par
If $n \geq 2$, then the tuple
\[
\left(x_1,\ldots,x_n\right)=\left(2^{\textstyle 2^{n-2}},2^{\textstyle 2^{n-3}},\ldots,256,16,4,2,1\right)
\]
has property {\em (4)}. Unfortunately, we do not know any explicitly given integers \mbox{$x_1,\ldots,x_n$}
with property {\em (4)} and \mbox{$|x_1|>2^{\textstyle 2^{n-2}}$}.

\noindent
Apoloniusz Tyszka\\
Technical Faculty\\
Hugo Ko\l{}\l{}\k{a}taj University\\
Balicka 116B, 30-149 Krak\'ow, Poland\\
E-mail address: \url{rttyszka@cyf-kr.edu.pl}
\end{document}